\documentclass[a4paper,12pt]{article}

\usepackage{amssymb}                                    
\usepackage{mathrsfs}                    
\usepackage{amsmath}                    
\usepackage{amsfonts}                   
\usepackage{amsthm}                     
\usepackage[latin1]{inputenc}           
\usepackage[arrow, matrix, curve]{xy}    
\usepackage[english]{babel}                
\usepackage{bbold} 			
\usepackage{nicefrac}
\usepackage{graphics}
\usepackage{hyperref}

\theoremstyle{plain}
\newtheorem{thm}{Theorem}[section]
\newtheorem{corollary}[thm]{Corollary}
\newtheorem{proposition}[thm]{Proposition}
\newtheorem{lemma}[thm]{Lemma}
\newtheorem{example}[thm]{Example}
\newtheorem{remark}[thm]{Remark}
\newtheorem{prop}[thm]{Proposition}
\newtheorem{variant}[thm]{Variant}

\theoremstyle{definition}
\newtheorem{definition}[thm]{Definition}

\newcommand\Hom		{\mathrm{Hom}}
\newcommand\dSet	{\text{dSet}}
\newcommand\Set     {\text{Set}}
\newcommand\sSet    {\text{sSet}}
\newcommand\Ab      {\mathrm{AbGr}}

\newcommand\calS    {\mathcal S}

\newcommand\Spec    {\calS\!\text{p}} 
\newcommand\Symm    {\mathcal{S}\text{ym}\mathcal{M}\text{on}\mathcal{C}\text{at}}
\newcommand\Symg    {\mathcal{S}\text{ym}\mathcal{M}\text{on}\mathcal{G}\text{rpd}}
\newcommand\Oper    {\text{Oper}}
\newcommand\sOper    {\text{sOper}}
\newcommand\indlim  {\underrightarrow{\lim}}
\newcommand\iso     {\stackrel{\sim}{\longrightarrow}}

\newcommand\calK    {\mathcal{K}}

\newcommand\EsSet   {E_\infty\text{-}\mathcal{S}\text{paces}}

\title{Algebraic K-Theory of $\infty$-Operads}
\author{Thomas Nikolaus}

\begin{document}

\maketitle

\begin{abstract} \noindent
The theory of dendroidal sets has been developed to serve as a combinatorial
model for homotopy coherent operads, see \cite{MoerW07, MC11}. An
$\infty$-operad is a dendroidal set $D$ satisfying certain lifting conditions.
 
In this paper we give a definition of K-groups $K_n(D)$ for a dendroidal set
$D$. These groups generalize the K-theory of symmetric monoidal (resp.
permutative) categories and algebraic K-theory of rings.  We establish some
useful properties like invariance under the appropriate equivalences and long
exact sequences which allow us to compute these groups in some examples. Using
results from \cite{Heuts2} and \cite{BaNi} we show that the $K$-theory groups of
$D$ can be realized as homotopy groups of a K-theory spectrum $\mathcal{K}(D)$. 
\end{abstract}

\section{Introduction}

Operads are an important tool in modern mathematics, especially in topology and
algebra \cite{marklOp}. Throughout this paper we use the term `operad' for what
should really be called a `coloured, symmetric operad' or maybe even better
a `symmetric multicategory'. In order to make clear what is meant let us briefly
recall that an operad $P$ in this sense is given by a set of colours
$\{a,b,c,...\}$, sets of operations $P(a_1,...,a_n ; b)$ equipped with $\Sigma_n$-actions, and composition maps.
Clearly classical non-coloured operads (i.e. coloured operads which have only
one colour) are an important special case. But there is another class of
examples which sit at the other end of the spectrum of coloured
operads; namely small symmetric monoidal categories $C$, which we
consider as operads with colours the objects of $C$ and operations
$C(a_1,...,a_n; b) := \Hom_C(a_1 \otimes ... \otimes a_n; b)$. 

The idea is that $\infty$-operads are both higher categorical and homotopy coherent 
versions of ordinary operads. There are several ways of making this idea precise.
The easiest model, which has been successfully used in topology for a long time
\cite{BoarVogt, May}, are topologically enriched operads, i.e. the sets
of operations $P(a_1,...,a_n; c)$ are replaced by topological spaces. But there
are other models which are technically more convenient. One model has been given
by Lurie \cite[Section 2]{HA} and another one by Moerdijk and Weiss \cite{MoerW07}. We
will restrict our attention in this paper to the latter model which goes by the
name of dendroidal sets. But all models are (at least conjecturally) equivalent,
so the results hold independently and should in principle have proofs in all
settings. We review the theory of dendroidal sets in Section \ref{sec_pre}. \\

In this paper we introduce abelian groups $K_n(D)$ for an $\infty$-operad $D$
which we call the K-theory groups of $D$. The zeroth group $K_0(D)$ can be defined
very explicitly using generators and relations, see Section \ref{sec_k0}. For
the higher groups we have to make use of homotopy theoretic methods. More
precisely we use a model structure on the category of dendroidal sets which was
introduced in \cite{BaNi}. By means of this model structure we can define for
every dendroidal set $D$ a `derived underlying space' whose homotopy groups are
the groups $K_n(D)$ (Section \ref{sec_higher}). 

We show that these groups are invariant under equivalences of $\infty$-operads
and that they admit long exact sequences coming from cofibre sequences of
$\infty$-operads. Using these properties one can already compute the $K$-groups
for basic cases. For example, for (a dendroidal version of) an
$E_\infty$-operad we show that $K_n(E_\infty) = 0$ for all $n$. Another easy
example of an $\infty$-operad is $\eta$, the trivial operad without higher operations; for this $\infty$-operad we
show that the $K$-groups are given by the stable homotopy groups of spheres.
More generally, we treat the example $\Omega[T]$, the $\infty$-operad associated to a tree $T$, and 
show that
$K_n\big(\Omega[T]\big) \cong \bigoplus_{\ell(T)}\pi_n^\mathcal{S}$, where $\ell(T)$ is the number 
of leaves of the tree $T$ and $\pi_n^\mathcal{S}$ are the stable homotopy groups of spheres (Corollary \ref{komega}).

It has been sketched by Heuts \cite{Heuts2} how to associate an infinite
loop space, i.e. a connective spectrum, to a dendroidal set $D$. We use a slight
variant of his construction to define what we call the algebraic K-theory
spectrum $\calK(D)$ of an $\infty$-operad (Section \ref{sec_spec}). We show that
the homotopy groups of this spectrum agree with our $K$-theory groups $K_n(D)$
(Theorem \ref{prop53}). In some cases we can identify this spectrum. For example
every simplicial set $X$ gives rise to a dendroidal set $i_!X$. For this case we
can show that the associated spectrum is the suspension spectrum of the
geometric realization of $X$. The main result of \cite{BaNi} even implies that
the functor $D \mapsto \calK(D)$ induces an equivalence between a suitable
localization of the category of $\infty$-operads and the homotopy category of
connective spectra. In particular all connective spectra arise as $\calK(D)$ for
some $\infty$-operad $D$. \\

Finally we want to explain why we have decided to call these invariants $K$-groups (and $K$-theory spectra). Recall that, by definition, algebraic
$K$-theory of a ring $R$ is computed using its category of finitely generated
projective modules (or some related space like $BGL(R)$). There are several
equivalent variants that produce the $K$-theory groups $K_n(R)$ and a $K$-theory
spectrum $\mathcal{K}(R)$ from this category. The whole theory was  initiated by
Quillen \cite{QuillenK}, but see also \cite{th82}. We have already explained in the
first paragraph that a symmetric monoidal category can be considered as an
operad and thus also as an $\infty$-operad. For the groupoid of finitely
generated projective modules over a ring $R$ this $\infty$-operad is denoted by
$N_d\text{Proj}_R$. We then show in Theorem \ref{generalization} that the
$K$-theory of this $\infty$-operad is equivalent to the algebraic $K$-theory of
the ring $R$:
\begin{equation*}
K_n(N_d\text{Proj}_R) \cong K_n(R) \quad \text{ and } \quad
\calK(N_d\text{Proj}_R) \cong \calK(R)\ .
\end{equation*}
In this sense our $K$-theory generalizes the algebraic $K$-theory of rings
(resp. symmetric monoidal categories) and therefore deserves to be called
$K$-theory. Moreover, it is shown in \cite{Heuts1} that not only ordinary
symmetric monoidal categories can be seen as dendroidal sets, but also symmetric
monoidal $\infty$-categories, i.e. $E_\infty$-algebras in the $\infty$-category of
$\infty$-categories (modeled by the Joyal model structure on simplicial sets).
In this sense the $K$-theory of dendroidal sets contains as a special case the
$K$-theory of symmetric monoidal $\infty$-categories.  \\

\noindent {\bf Acknowledgements.} The author would like to thank Matija
Ba\v{s}i\'{c}, David Gepner, Gijs Heuts and Ieke Moerdijk for helpful
discussions and Peter Arndt, Matija Ba\v{s}i\'{c} and David Gepner for comments on the draft.

\section{Preliminaries about dendroidal sets}\label{sec_pre}

In this section we  briefly recall the theory of dendroidal sets as
discussed in \cite{MoerW07, MW09, MC10, MC11, MC11a}. Dendroidal sets are
a generalization of simplicial sets. Recall that simplicial sets are
presheaves on the category $\Delta$ of finite, linearly ordered sets, i.e. $\sSet =
[\Delta^{op}, \text{Set}]$. 
Dendroidal sets are similarly a presheaf category on an extension of the simplicial category $\Delta$ 
to the dendroidal category $\Omega$.

The category $\Omega$ is defined as follows. The objects are given by finite,
rooted trees; that is, graphs with no loops equipped with a distinguished outer
edge called the root and a (possibly empty) set of outer edges not containing
the root called leaves. As an example consider the tree:  \\
\begin{equation} \label{exampletree}
T= 
\begin{minipage}{2.5cm}

\xymatrix@R=10pt@C=12pt{
 & &  & \\
& *=0{~~\bullet_{~v}}  \ar@{-}[ul]^-a \ar@{-}[ur]^-b && \\
&&  *=0{~~\bullet_w} \ar@{-}[ul]^-c \ar@{-}[ur]^-d & \\
&&  \ar@{-}[u]^e &
}

\end{minipage} 
\end{equation}

We also allow the tree which only consist of a single edge without vertices. For this tree the root is also a leaf, in contrast to the  requirement that the leaves are distinct from the root. In particular, this tree should be carefully distinguished from the other tree with one edge
$$T = \xymatrix@R=10pt@C=12pt{
*=0{\bullet}  \ar@{-}[d]^a  \\
~
} $$
which has no leaves.

Every finite, rooted tree $T$ generates a coloured symmetric operad  $\Omega(T)$ as follows. The set of colours is given by the edges
of $T$. In Example \eqref{exampletree} this is the set $\{a,b,c,d,e\}$. The operations are
freely generated by the vertices of $T$. In Example \eqref{exampletree} there is one generating
operation $v \in \Omega(T)(a,b;c)$ and another one $w \in \Omega(T)(c,d;e)$.
These operations of course generate other operations such as $w \circ_{c} v \in
\Omega(T)(a,b,d;e)$, the permutations $\sigma v \in \Omega(T)(b,a;c)$, $\sigma w
\in \Omega(T)(d,c;e)$, and the six possible permutations of $w \circ_{c} v $.

Now we can complete the definition of the category $\Omega$, whose objects are all finite rooted trees $T$, by specifying the morphisms.  A morphism $T \to S$  in $\Omega$ is a morphism of 
coloured symmetric operads $\Omega(T)\to \Omega(S)$ (for the definition of operad morphisms see e.g. \cite[Section 2]{MoerW07}). Thus, the category $\Omega$ is, by definition, a full subcategory of the category of coloured operads. Examples of
morphisms are shown in the next picture: \\
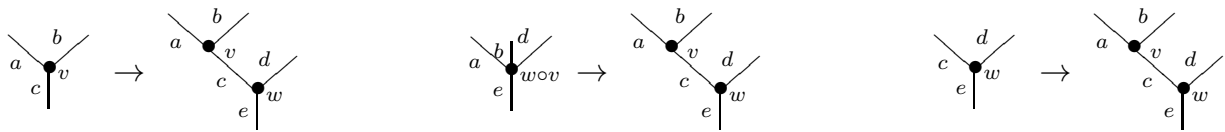
\begin{minipage}{1.5cm}
$
\xymatrix@R=10pt@C=12pt{
 &  & \\
& *=0{\,\,\,\bullet_v}  \ar@{-}[ul]^-a \ar@{-}[ur]^-b&  \\
&   \ar@{-}[u]^-c  & \\
}
$
\end{minipage}
$\rightarrow$
\begin{minipage}{2.5cm}
$
\xymatrix@R=10pt@C=12pt{
 & &  & \\
& *=0{~\,\,\bullet_{\,\,v}}  \ar@{-}[ul]^-a \ar@{-}[ur]^-b && \\
&&  *=0{~~\bullet_w} \ar@{-}[ul]^-c \ar@{-}[ur]^-d & \\
&&  \ar@{-}[u]^e &
}
$
\end{minipage}
\hspace{1.3cm}
\begin{minipage}{1.5cm}
$
\xymatrix@R=10pt@C=12pt{
 &  & \\
& *=0{\quad~ \bullet_{w \circ v}}  \ar@{-}[ul]^a \ar@{-}[ur]^d \ar@{-}[u]^{ b}& 
\\
&   \ar@{-}[u]^-e  & \\
}
$
\end{minipage}
$\rightarrow$
\begin{minipage}{2.5cm}
$
\xymatrix@R=10pt@C=12pt{
 & &  & \\
& *=0{~\,\,\bullet_{\,\,v}}  \ar@{-}[ul]^-a \ar@{-}[ur]^-b && \\
&&  *=0{~~\bullet_w} \ar@{-}[ul]^-c \ar@{-}[ur]^-d & \\
&&  \ar@{-}[u]^e &
}
$
\end{minipage}
\hspace{1.3cm}
\begin{minipage}{1.5cm}
$
\xymatrix@R=10pt@C=12pt{
 &  & \\
& *=0{~~\bullet_w}  \ar@{-}[ul]^-c \ar@{-}[ur]^-d&  \\
&   \ar@{-}[u]^-e  & \\
}
$
\end{minipage}
$\rightarrow$
\begin{minipage}{2.5cm}
$
\xymatrix@R=10pt@C=12pt{
 & &  & \\
& *=0{~\,\,\bullet_{\,\,v}}  \ar@{-}[ul]^-a \ar@{-}[ur]^-b && \\
&&  *=0{~~\bullet_w} \ar@{-}[ul]^-c \ar@{-}[ur]^-d & \\
&&  \ar@{-}[u]^e &
}
$
\end{minipage}
Here the labeling of the edges and vertices in the domain trees indicates to which edges and vertices in the target trees they are sent. The depicted morphisms are all examples of a special class
of morphisms in $\Omega$ called \textit{face maps}. There are two types of face
maps in $\Omega$. The first type are the
outer face maps, which are obtained by chopping off an
outer vertex of a tree $T$. The first and the third
morphism in the picture are outer face maps. The second type of face maps are the inner face maps, which are obtained by contracting
an inner edge of $T$. An example is the morphism in the middle of the picture which is
obtained by contracting the edge $c$. Every tree has a set of \emph{outer face
maps} which are labeled by outer vertices $v$ and a set of \emph{inner face maps} which
are labeled by inner edges $e$. \\

\noindent
The category $\dSet$ of dendroidal sets is  defined as the presheaf category on
$\Omega$, i.e.
\begin{equation*}
\dSet := [\Omega^{op},\Set].
\end{equation*}
For a dendroidal set $D$ we denote the value on the tree $T$ by $D_T$ and call it the set
of $T$-dendrices. The dendroidal set represented by a tree $T$ is denoted by
$\Omega[T]$. In particular for the tree with one edge and no vertex we set $\eta:= \Omega[~|~]$. The Yoneda lemma shows that we have $D_T
\cong \Hom(\Omega[T], D)$.

There is a fully faithful embedding of the simplex category $\Delta$ into
$\Omega$ by considering finite linear ordered sets as linear trees. More
precisely this inclusion maps the object $\{0,1,...,n\}=[n] \in \Delta$ to the
tree 
\begin{equation*}
L_n= 
\begin{minipage}{2.5cm}

\xymatrix@R=10pt@C=12pt{
 \\
*=0{\bullet}  \ar@{-}^{a_0}[u] \\
*=0{\bullet}  \ar@{-}^{a_1}[u] \\
*=0{...}  \ar@{-}[u] \\
*=0{\bullet}  \ar@{-}[u] \\
 \ar@{-}^{a_n}[u]
}

\end{minipage} 
\end{equation*}
This inclusion $\Delta \subset \Omega$ induces an adjunction
\begin{equation*}
i_!: \xymatrix{\sSet \ar@<0.3ex>[r] & \dSet: i^* \ar@<0.7ex>[l]}
\end{equation*}
with fully faithful left adjoint $i_!$ (there is also a further right adjoint
$i_*$ which does not play a role in this paper). The functor $i^*$ is given by
restriction to linear trees and the functor $i_!$ is extension by zero, i.e. the
dendroidal set $i_!X$ agrees with $X$ on linear trees and is empty otherwise.

The inclusion of $\Omega$ into the category of coloured, symmetric operads
induces a fully faithful functor $N_d: \Oper \to \dSet$ called the
\emph{dendroidal nerve}. Concretely the dendroidal nerve of an operad $P$ is
given by $N_d(P)_T := \Hom(\Omega(T),P)$.
By definition of $\Omega$ and the Yoneda lemma we have $N_d(\Omega(T)) =
\Omega[T]$. A particularly important case of the dendroidal nerve is when 
the operad $P$ comes from a
(small) symmetric monoidal category $C$. Recall that the operad
associated to $C$ has as colours the objects of $C$ and as operations
$\Hom(c_1,...,c_n; c) = \Hom_C(c_1 \otimes ... \otimes c_n; c)$. By abuse of
notation we call the nerve of this operad $N_d(C)$ as well. This
assignment defines a fully faithful inclusion 
\begin{equation*}
\Symm \to \dSet
\end{equation*}
where $\Symm$ denotes the category of (small) symmetric monoidal categories and
lax monoidal functors. Here lax monoidal for a functor $F$ means that the
structure morphisms $F(a) \otimes F(b) \to F(a \otimes b)$ and $\mathbb{1} \to F(\mathbb{1})$ are not
necessarily invertible.
\\

\noindent
As described above, every tree $T$ has a set of subobjects called faces, which
are obtained by either contracting an inner edge (these are the inner faces) or by truncating outer edges (these are the outer faces). These faces are used to define
dendroidal $\textit{horns}$ and $\textit{boundaries}$ in a way that generalizes horns
and boundaries in simplicial sets. The boundary 
\begin{equation*}
\partial \Omega[T] \subset \Omega[T]
\end{equation*}
is a subobject of $\Omega[T]$ which is the union of all inner and outer faces of
$T$. The horns of $T$ carry a label $a$ which can be either an inner edge or an
outer vertex. Then the horn 
\begin{equation*}
\Lambda^a[T] \subset \Omega[T]
\end{equation*}
is defined as the union of all faces of $T$ except the face labelled by $a$. We
call the horn inner if $a$ is an inner edge and outer if $a$ is an outer vertex.
The reader can easily convince himself that for the case $T= L_n$ these horns
reduce to the horns of the simplex $\Delta[n]$. 

There is one case which deserves special attention, namely the case of trees with 
exactly one vertex. These trees are called \emph{corollas}. More concretely the
$n$-corolla is given by
\begin{equation*}
C_n = 
\begin{minipage}{2.5cm}

\xymatrix@R=10pt@C=12pt{
a_1 & a_2 & ... &  & a_n \\
& & *=0{\bullet}  \ar@{-}[ull] \ar@{-}[ul] \ar@{-}[u] \ar@{-}[ur]   \ar@{-}[urr]
&& \\
&& {b} \ar@{-}[u] & &
}

\end{minipage}
\end{equation*}
We observe that $C_n$ has $n+1$ faces given by the inclusion of the colours
$a_1,...,a_n,b$. All faces are outer. Consequently $C_n$ also has $n+1$ outer
horns which are inclusions of $n$-disjoint unions of $\eta$. 

\begin{definition}\label{def_inner}
A dendroidal set $D$ is called \emph{inner Kan} or an \emph{$\infty$-operad} if $D$
admits fillers for all inner horns, i.e. for each morphism $\Lambda^{e}[T] \to
D$ with $e$ an inner edge there is a morphism $\Omega[T] \to D$ that renders
the diagram
\begin{equation*}
\xymatrix{
\Lambda^{e}[T] \ar[d]\ar[r] & D \\
\Omega[T] \ar[ru] &
}
\end{equation*}
commutative. A dendroidal set is called \emph{fully Kan} if it admits fillers for all
horns.
\end{definition}

The two classes of inner Kan and fully Kan dendroidal sets are very important in
the theory of dendroidal sets and for the rest of the paper. Thus we make some
easy remarks such that the reader gets a feeling for these classes:
\begin{remark}
\begin{itemize}
\item Let $X$ be a simplicial set. Then $i_!X$ is an $\infty$-operad if and only
if $X$ is an $\infty$-category (in the sense of Boardman-Vogt, Joyal, and Lurie).
Moreover for every $\infty$-operad $D$ the underlying simplicial set $i^*D$ is
an $\infty$-category.
\item For a fully Kan dendroidal set $D$ the underlying simplicial set $i^*D$ is
a Kan complex, but for a non-empty simplicial set $X$ the dendroidal set $i_!X$
is never fully Kan since there are no fillers for corolla horns.
\item For every (coloured, symmetric) operad $P$ the dendroidal nerve $N_dP$ is
an $\infty$-operad. This shows that ordinary operads are a special case of
$\infty$-operads. In particular all representable dendroidal sets $\Omega[T]$ are $\infty$-operads.
\item The dendroidal nerve $N_dP$ is fully Kan if and only if $P$ comes from a
grouplike symmetric monoidal groupoid $C$. These are also called
Picard-groupoids. In \cite{BaNi} it has been shown that fully Kan dendroidal
sets model Picard-$\infty$-groupoids. Therefore the last statement shows that
Picard-groupoids are a special case of Picard-$\infty$-groupoids.
\end{itemize}
\end{remark}

The foundational result in the theory of $\infty$-operads is that there is a
model structure on the category of dendroidal sets with fibrant objects given
by $\infty$-operads. This
model structure is a generalization of the the Joyal model structure on
simplicial sets. In order to state the result properly we have to introduce the
class of \emph{normal monomorphisms} of dendroidal sets. This class is defined as
the smallest class of morphisms in $\dSet$ that contains the boundary inclusions
of trees and that is closed under pushouts, retracts and transfinite
compositions. One can also give an easy concrete description (see
\cite[2.3]{MC11}) but we will not need this description here.
\begin{thm}[Cisinski-Moerdijk]
There is a model structure on the category of dendroidal sets with cofibrations
given by normal monomorphisms and fibrant objects given by $\infty$-operads. 
\end{thm}
Note that a model structure is uniquely determined by its class of cofibrations
and the fibrant objects. Thus the above result can be read as an existence
statement. The weak equivalences in the Cisinski-Moerdijk model structure are
called \emph{operadic equivalences}. Their definition is not quite explicit, though one can give explicit criteria for a morphisms between
$\infty$-operads to be an operadic equivalence, see \cite[Theorem 3.5 \& Theorem
3.11]{MC10}.
A similar model structure exists for fully Kan dendroidal sets \cite{BaNi}:
\begin{thm}
There is a model structure on the category of dendroidal sets with cofibrations
given by normal monomorphisms and fibrant objects given by fully Kan dendroidal
sets. 
\end{thm}
We call this second model structure the \emph{stable} model structure and the
weak equivalences \emph{stable equivalences}. The stable model structure is a
left Bousfield localization of the Cisinski-Moerdijk model structure, i.e. every
operadic equivalence is also a stable equivalence. Again stable equivalences are
defined very indirectly, but we will give a more direct criterion in Proposition
\ref{prop:hk}.

The existence of the stable  model structure also implies that for every
dendroidal set $D$ we can choose a functorial fibrant replacement $D \to D_K$
with $D_K$ being fully Kan. This can for example be done using Quillen's small
objects argument by iteratively gluing in fillers for horns. This fibrant
replacement will play an important role in our definition of the K-theory groups
of $D$.  \\

\noindent Finally we want to remark that the adjunction $i_!: \xymatrix{\sSet
\ar@<0.3ex>[r] & \dSet: i^* \ar@<0.7ex>[l]}$ becomes a Quillen adjunction in the
following two cases:
\begin{itemize}
\item For the Kan-Quillen model structure on $\sSet$ and the stable model
structure on $\dSet$.
\item For the Joyal model structure on $\sSet$ and the Cisinski-Moerdijk model
structure on $\dSet$.
\end{itemize}

\section{\texorpdfstring{K$_0$}{K0} of dendroidal sets}\label{sec_k0}

In this section we want to define an abelian group $K_0(D)$ for each dendroidal
sets $D$ which generalizes the group $K_0(C)$ for a symmetric monoidal category
$C$.
The latter group $K_0(C)$ is defined as the group completion of the abelian
monoid
$\pi_0(C)$. 

Note that there are two possible meanings for $\pi_0(C)$ of a symmetric
monoidal category $C$. It can either be the set of isomorphism classes in $C$ or
the connected components in $NC$. In the case that $C$ is a groupoid the two
choices agree. Here we want $\pi_0(C)$ to mean the connected components. In
particular to recover the classical $K$-theory of a ring $R$ we have to compute
$K_0(\text{Proj}_R)$ where $\text{Proj}_R$ denotes the maximal subgroupoid
inside the category of finitely generated projective $R$-modules.

\begin{definition}\label{def:knull}
Let $D$ be a dendroidal set. We define $K_0(D)$ as the abelian group freely
generated
by the elements $x \in D_{L_0}$ (i.e. morphisms $\eta \to D$)
subject to the relations
\begin{equation*}
 x_1 + ... + x_n = x 
\end{equation*}
whenever there is a corolla $\Omega[C_n] \to D$ with ingoing faces $x_1,...,x_n$
and outgoing face $x$. For a $0$-corolla $\Omega[C_0] \to D$ the left hand sum
is understood to be 0. A map $f: D \to D'$ of dendroidal sets induces a morphism
$f_*: K_0(D) \to K_0(D')$ of abelian groups by applying $f_{L_0}$.
\end{definition}
Note that the relations we impose in the definition of $K_0(D)$ also include the
possibility of the 1-corolla $C_1 = L_1$. Hence if two elements $x,y \in
D_{L_0}$ are equal in $\pi_0(i^*(D))$, i.e. there is a chain of edges in the
underlying simplicial set connecting them, then they are also equal in
$K_0(D)$. Therefore we have a well defined map (of sets)
\begin{equation*}
\lambda_D: \quad \pi_0(i^*D)) \longrightarrow K_0(D). 
\end{equation*}
\begin{proposition}\phantomsection\label{prop:k}
\begin{enumerate}
\item 
If $D = i_!X$ for a simplicial set $X$, then $K_0(D)$ is freely generated by
$\pi_0(i^*D) = \pi_0(X)$.
\item
If $D = N_dC$ for a symmetric monoidal category $C$, then $\pi_0(i^*D) =
\pi_0(NC)$ is an abelian monoid and $\lambda_D$ exhibits $K_0(D)$ as the group
completion of $\pi_0(NC)$. In particular we have 
$K_0(D) = K_0(C)$.
\item \label{item:eins}
If $D$ is fully Kan, then $\lambda_D$ is a bijection.
\end{enumerate}
\end{proposition}
\begin{proof}
1) For a simplicial set $X$ there are no $n$-corollas $C_n \to i_!(X)$ except
for $n = 1$. Hence $K_0(i_!(X))$ is the free group generated by the vertices of
$X$ where two vertices are identified if there is an edge connecting them. The
usual description of $\pi_0(X)$ then proves the statement. \\

2) The existence of the monoid structure on $\pi_0(NC)$ and the fact that
$\lambda_D$ is a group homomorphism are immediate from the definition and the
fact that an $n$-corolla in $N_d(C)$ is exactly a morphisms $c_1 \otimes ...
\otimes c_n \to c$ in $C$. The claim then follows from the fact that the group
completion of an abelian monoid $A$ can be described as the free group generated
by objects $a \in A$ subject to the relations $a_1 + ... + a_n = a$ whenever
this holds in $A$. \\

3) We want to explicitly construct an inverse to $\lambda_D$. But as a first
step we endow $\pi_0(i^*D)$ with the structure of an abelian group. Therefore
for two elements $a, b \in (i^*D)_0 = D_{L_0}$ we choose a 2-corolla 
$\Omega[C_2] \to D$ with inputs $a$ and $b$ and output $c$. Such a corolla
exists since $D$ is fully Kan. Then we set $[a] + [b] := [c]$. 

We have to show that this addition is well defined. So first assume that there
is another $a' \in D_{L_0}$ with $[a'] = [a]$ in $\pi_0(i^*D)$. This means that
there is an 
edge $L_1 \to D$ connecting $a'$ and $a$. We look at the tree 
$$  
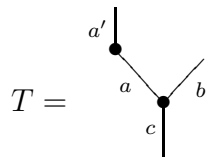
\xymatrix@R=10pt@C=12pt{
& *=0{} &&&&&& \\
&  *=0{\bullet} \ar@{-}[u]^{a'} &&&&&& \\
T= &&*=0{\bullet} \ar@{-}[ul]^{a} \ar@{-}[ur]_{b} &&&&&\\
&&*=0{}\ar@{-}[u]^{c} &&&&& \\
}
$$
The two maps $\Omega[C_2] \to D$ and $L_1 \to D$ chosen above determine a map
from the inner horn $\Lambda^a[T] \to D$. Since $D$ is fully Kan we can fill
this horn and obtain a map $\Omega[T] \to D$. In particular there is a 2-corolla
$\Omega[C_2] \to D$ with inputs $a', b$ and output $c$. Thus we have $[a'] + [b]
= [c] = [a] + [b]$. This shows that the addition does not depend on the
representatives of $[a]$ and by symmetry also for $[b]$. 

It remains to check that the addition does not depend on the choice of 2-corolla
$\Omega[C_2] \to D$ with inputs $a$ and $b$. Therefore assume we have another
2-corolla in $D$ with inputs $a,b$ and output $c'$. Then consider the tree
$$  
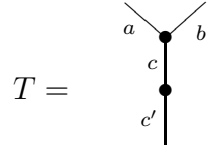
\xymatrix@R=10pt@C=12pt{
&&&&&&& \\
&&*=0{\bullet} \ar@{-}[ul]^{a} \ar@{-}[ur]_{b} &&&&&\\
T= &&*=0{\bullet}\ar@{-}[u]^{c} &&&&& \\
&&*=0{}\ar@{-}[u]^{c'} &&&&&
}
$$
As before we fill this tree at the outer horn of the binary vertex.
This yields an edge between $c$ and $c'$ and thus shows that the addition is well-defined. 

Before we show that $\pi_0(i^*D)$ together with the addition is really an
abelian group we need another preparatory fact: in $\pi_0(i^*D)$ the equality
$\big((...((a_1 + a_2) + a_3) + ... )+ a_n\big)  = b$ holds precisely if there
is an $n$-corolla $\Omega[C_n] \to D$ with inputs $a_1,...,a_n$ and output $b$. 
We show this by induction. For $n=2$ this is the definition. Assume it holds for
$n$. Then the claim for $n+1$ follows straightforward by looking at horns of the
tree obtained by grafting an $n$-corolla onto a 2-corolla.

Now we need to show that $\pi_0(i^*D)$ together with the addition operation just defined  is in fact an
abelian group. The fact that the multiplication is abelian is automatic by the
fact that we consider symmetric operads (resp. non-planar trees). So we need to
show that it is associative, there are inverses and units. This follows be
filling the root horns $\eta_a \sqcup \eta_b \to \Omega[C_2]$ and $\emptyset \to
\Omega[C_0]$ and we leave the details to the reader. 

Altogether we have shown that $\pi_0(i^*D)$ admits the structure of an abelian
group. By definition of the two group structures the morphism $\lambda_D:
\pi_0(i^*D) \to K_0(D)$ is a group homomorphism. An inverse is now induced by
the map $D_{L_0} \to \pi_0(i^*D)$ and the fact that $K_0(D)$ is freely generated
by $D_{L_0}$ subject to relations which, as shown above, hold in $\pi_0(D)$.
\end{proof}

\begin{example}\phantomsection\label{exampleknull}
\begin{itemize}
\item
For a tree $T$ let $\ell_T$ denote the set of leaves of $T$. Then we have 
$$K_0\big(\Omega[T]\big) \cong \mathbb{Z}\langle \ell_T \rangle\ ,$$
i.e. $K_0$ is the free abelian group generated by the set of leaves of $T$. 
\item
For a map $f: \Omega[S] \to \Omega[T]$ in $\Omega$ we have for each edge $e \in
\ell_S$ the subset $\ell_{f(e)} \subset \ell_T$ of leaves over the edge $f(e)$.
The induced map
$$ K_0(S) \cong \mathbb{Z}\langle \ell_S \rangle \to  K_0(T) \cong
\mathbb{Z}\langle \ell_T \rangle$$
is then the map that sends the generator $e \in \mathbb{Z}\langle \ell_S
\rangle$ to the element
$e_1 + ... + e_k \in \mathbb{Z}\langle \ell_T \rangle$ where $\{e_1,...,e_k\} =
\ell_{f(e)}$ is the set of leaves over $f(e)$.
\end{itemize}
\end{example}

\begin{lemma}\phantomsection\label{adjoint}
\begin{enumerate}
\item
The functor $K_0: \dSet \to \Ab$ is a left adjoint and thus preserves all
colimits. 
\item 
For a horn inclusion $\Lambda^a[T] \to \Omega[T]$ the induced morphism 
$$ K_0\big(\Lambda^a[T]\big) \to K_0\big(\Omega[T]\big)$$
is an isomorphism.
\end{enumerate}
\end{lemma}
\begin{proof}
1) $K_0$ is left adjoint to the inclusion functor $i: \Ab \to \dSet$ which can
be described as follows. Let $A$ be an abelian group. We consider it as a discrete
symmetric monoidal category (i.e. without non trivial morphisms) then we consider
it as an operad (as described in the previous section) and take the dendroidal
nerve. Explicitly we obtain $i(A)_T = A \langle\ell_T\rangle$. Then it is easy
to see that $K_0$ is left adjoint to $i$.

2) If the tree $T$ has more than two vertices one easily verifies that the
inclusion $\Lambda^a[T] \to \Omega[T]$ induces a bijection when evaluated on
$L_0$ and on $C_n$. Thus it clearly induces an isomorphism on $K_0$. Therefore
it remains to check the claim for horns of trees with one or two vertices. For
one vertex the tree is a corolla and the horn is a disjoint union of $\eta$'s.
Then the verification of the statement is straightforward using Example
\ref{exampleknull}. Thus only the case of trees with two vertices remains. 

Such trees can all be obtained by grafting an $n$-corolla $C_n$ for $n \geq 0$ on
top of a $k$-corolla for $k \geq 1$. We call this tree $C_{n,k}$.
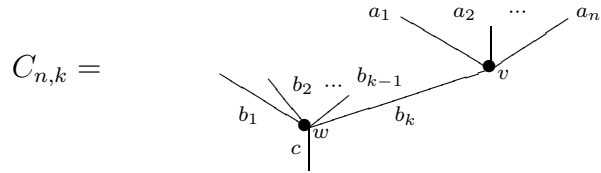
$$  \xymatrix@R=10pt@C=12pt{
&&&&&&&\\
&&&&_{a_1}&_{a_2 \quad\cdots}&_{a_n}& \\
C_{n,k}= \qquad && &&_{b_{k-1}}& *=0{~\bullet_v}
\ar@{-}[ul]\ar@{-}[u]\ar@{-}[ur]\ar@{-}&& \\
&&&*=0{~\bullet_w}\ar@{-}[ull]^{b_1}\ar@{-}[ul]_{b_2 ~
\cdots}\ar@{-}[ur]\ar@{-}[urr]_{b_k}&&&&\\
&&&*=0{}\ar@{-}[u]^c &&&&
}
$$
There are three possible horns and applying the definitions yields the following
groups:
\begin{align*}
K_0(\Lambda^{b_k}[C_{n,k}])&= \frac{ \mathbb{Z} \langle a_1,...,a_n,b_1,...,b_k,
c \rangle}{a_1 + ... + a_n = b_k, ~ b_1 + ... +b_k = c} \\
K_0(\Lambda^{v}[C_{n,k}]) &= \frac{ \mathbb{Z} \langle a_1,...,a_n,b_1,...,b_k,
c \rangle}{a_1 + ... + a_n = b_k, ~ b_1 + ... +b_{k-1} + a_1 + ... + a_n = c} \\
K_0(\Lambda^{w}[C_{n,k}]) &= \frac{ \mathbb{Z} \langle a_1,...,a_n,b_1,...,b_k,
c \rangle}{b_1 + ... +b_k = c, ~ b_1 + ... +b_{k-1} + a_1 + ... + a_n = c} 
\end{align*}
Clearly these groups are all isomorphic to $K_0(\Omega[C_{n,k}]) \cong
\mathbb{Z} \langle a_1,...,a_n, b_1,...,b_{k-1} \rangle$.
\end{proof}

\begin{proposition}\label{prop:knull}
For a stable equivalence $f: D \to D'$ the induced morphism $f_*: K_0(D) \to
K_0(D')$ is an isomorphism. In particular this holds also for operadic
equivalences.
\end{proposition}
\begin{proof}
We first show that for $D$ a dendroidal set and $D_K$ the fibrant replacement
obtained by Quillen's small object argument the induced morphism $K_0(D)\to
K_0(D_K)$ is an isomorphism. Therefore remember that $D_K$ is built as the
directed colimit
$$ D_K = \indlim( D_0 \to D_1 \to D_2 \to ... )$$
where $D_0 = D$ and each $D_{n+1}$ is obtained by attaching trees along horns to
$D_n$. Since $K_0$ is left adjoint by Lemma \ref{adjoint} we have 
$$ K_0(D_K) \cong \indlim( K_0(D_0) \to K_0(D_1) \to K_0(D_2)\to ... ). $$
Hence it suffices to check that each group homomorphism $K_0(D_n) \to
K_0(D_{n+1})$ is an isomorphism. To see this note that
$D_{n+1}$ is build as a pushout of the form
$$\xymatrix{
\bigsqcup \Lambda^a[T] \ar[r]\ar[d] & D_n\ar[d]  \\
\bigsqcup \Omega[T] \ar[r] & D_{n+1}
}$$
Applying $K_0$ yields a pushout diagram
$$\xymatrix{
\bigoplus K_0(\Lambda^a[T]) \ar[r]\ar[d] & K_0(D_n)\ar[d] \\
\bigoplus K_0(\Omega[T]) \ar[r] & K_0(D_{n+1})
}$$
where the left vertical morphism is, by Lemma \ref{adjoint}, an
isomorphism. Thus the right vertical morphism is an isomorphism too. Altogether
this shows that the morphism $K_0(D)\to K_0(D_K)$ is an isomorphism of groups. 

Now assume we have an arbitrary stable equivalence $f: D \to D'$ of dendroidal
sets. Applying the fibrant replacement described above to both objects yields a
weak equivalence $f_K: D_K \to D_K'$. By the above argument it remains only to
check that the induced morphism $(f_K)_*: K_0(D_K) \to K_0(D_K')$ is an
isomorphism. By the fact that $i^*$ is right Quillen we know that $f_K$ induces
a weak equivalence of simplicial sets $i^*(D_K) \to i^*(D_K')$. Together with
Proposition \ref{prop:k} \eqref{item:eins} this shows that $(f_K)_*: K_0(D_K)
\to K_0(D_K')$ is an isomorphism, hence  $f_*: K_0(D) \to K_0(D')$ is too.

Finally the statement about operadic equivalences follows since each operadic
equivalence is a stable equivalence.
\end{proof}

\section{Higher \texorpdfstring{$K$}{K}-groups of dendroidal
sets}\label{sec_higher}

In this section we want to define higher $K$-groups of a dendroidal set $D$.
Note that by the results of
the last section we can compute $K_0(D)$ for a dendroidal set $D$ always as 
follows: choose a fully Kan replacement $D_K$ of $D$. Then by Proposition
\ref{prop:knull} the inclusion $D \to D_K$ induces an 
isomorphism $K_0(D) \iso K_0(D_K)$. But by Proposition
\ref{prop:k}\eqref{item:eins} we can compute $K_0(D_K)$ as the connected components of the underlying space 
$i^*(D_K)$. The idea for the higher $K$-groups is to generalize this procedure. 
\begin{definition}\label{def:higher}
Let $D$ be a dendroidal set and $D_K$ be a fully Kan replacement. We define 
$$ K_n(D) := \pi_n(i^*D_K). $$
The space $i^*D_K$ is also called the \emph{derived underlying space} of $D$.
\end{definition}

\begin{remark}
\begin{itemize}
\item
There is a subtlety involved in the above definition, namely the choice of
basepoint in $i^*D_K$ to compute the higher homotopy groups. It turns out that
the choice of basepoint is inessential since $i^*D_K$ is an infinite loop space
as we will see later. 
\item The fact that $i^*D_K$ is an infinite loop space also implies that all the
$K_n(D)$ are abelian groups (especially $K_0$ and $K_1$).
\item
Note that in order to turn $K_n$ into functors $\dSet \to \Ab$ we have to make
functorial choices of fibrant replacements. This can, e.g., be done using
Quillen's small object argument. This  also solves the problem of basepoints, 
since then
the fibrant replacement $D_K$ has a distinguished morphism $\Omega[C_0] \to D_K$
coming from gluing in $\Omega[C_0]$ along its outer horn $\emptyset$.  After
applying $i^*$ this leads to a (functorial) choice of basepoint $\Delta[0] \to
i^*D_K$. 
\end{itemize}
\end{remark}

\begin{proposition}\phantomsection\label{prop:hk}
\begin{enumerate}
\item Definition \ref{def:knull} and \ref{def:higher} of $K_0$ agree, i.e the
groups are canonically isomorphic. 
\item The higher $K$-groups are well-defined, i.e. for two fully Kan
replacements $D_K$ and $D'_K$ of $D$ there is an  isomorphism $\pi_n(i^*D_K)
\iso \pi_n(i^*D'_K)$.
\item For a stable equivalence between dendroidal sets the induced morphisms on
$K$-groups are isomorphisms. In particular, for stably equivalent dendroidal
sets the $K$-groups are isomorphic.

Conversely if for a morphism $f: D \to D'$ the induced maps $f_*: K_n(D) \to
K_n(D')$ are isomorphisms then $f_*$ is a stable equivalence.
\end{enumerate}
\end{proposition}
\begin{proof}
1) This follows directly from our remarks preceding Definition \ref{def:higher}.

2) The two fibrant replacements $D_K$ and $D_K'$ are related by the chain $D_K
\leftarrow D \rightarrow D_K'$ of weak equivalences. As a lift in the diagram
$$\xymatrix{
D \ar[d] \ar[r] & D_K' \ar[d]\\
D_K \ar[r] & {*}
}$$
we find a morphism $f: D_K \to D_K'$ which is by the 2-out-of-3 property a
stable equivalence. By the fact that $i^*$ is right Quillen this implies that
$i^*f: i^*D_K \to i^*D_K'$ is a weak equivalence of simplicial set. Therefore
$f_*: \pi_n\big(i^*D_K\big) \to \pi_n\big(i^*D_K'\big)$ is an isomorphism.

3) The morphism $f: D \to D'$ induces a morphism between fibrant replacements
$f_K: D_K \to D_{K'}$. Then $f$ is a stable equivalence if and only if $f_K$ is
a stable equivalence. By \cite[Theorem 4.2 (3)]{BaNi} this is equivalent to the
fact that $i^*f_K: i^*D_K \to i^*D_{K'}$ is a weak equivalence of simplicial
sets. And this last condition is equivalent to the fact that the induced
morphisms on homotopy groups, which are the $K$-groups of $D$ and $D'$, are
isomorphisms. 
\end{proof}

\begin{corollary}
For an operadic equivalence $f: D \to D'$ of dendroidal sets, the induced
morphisms
$f_*: K_n(D) \to K_n(D')$ are isomorphisms. Hence the $K$-groups are an
invariant of $\infty$-operads.
\end{corollary}
\begin{proof}
This follows from the fact, that the stable model structure is a left Bousfield
localization of the Moerdijk-Cisinski model structure. So the weak operadic
equivalences are also stable equivalences.
\end{proof}

\begin{example}
Let $\mathcal{E}_\infty$ in $\dSet$ be a cofibrant resolution of the point $*
\in \dSet$. This is the dendroidal version of an $E_\infty$-operad. Then we have 
$K_n(\mathcal{E}_\infty) \cong K_n(*) \cong 0$ for all $n$.

There are also dendroidal versions $\mathcal{E}_k$ of the little $k$-disks operad. One can show that we also have $K_n(\mathcal{E}_k) \cong 0$ for
all $n$. We will not do this here, since it is most easily deduced using
monoidal properties of $K$-theory which will be investigated elsewhere.
\end{example}

\begin{proposition}
Let $D = \indlim~ D_i$ be a filtered colimit of dendroidal sets $D_i$. Then we
have
$$ K_n(D) \cong \indlim~ K_n\big(D_i\big).$$
\end{proposition}
\begin{proof}
We will use the fact that there is an endofunctor $T: \dSet \to \dSet$ with the
property that when applied to a dendroidal set $D$ it produces a fibrant replacement
$T(D)$ and furthermore $T$ preserves filtered colimits. 
To see that such a $T$ exists one can use the fact that $\dSet$ is combinatorial
and use a general existence result for accessible fibrant replacement functors (see, e.g.,
\cite[Proposition A.1.2.5]{HTT}) or 
use \cite[Proposition 2.23]{AlgKan} for an explicit construction.

Now for a given filtered colimit $D = \indlim D_i$ we obtain a replacement 
$T(D) = \indlim T\big(D_i)$. Then note that the
 functor $i^*: \dSet \to \sSet$ is not only right adjoint to $i_!$ but also left
adjoint to a functor $i_*$. Hence it also preserves filtered colimits. 
Therefore we have $i^*T(D) \cong \indlim~ i^*T(D_i)$
which implies 
\begin{center}
$K_n(D) = \pi_n\big(i^*T(D)\big) \cong \indlim~\pi_n\big(i^*T(D_i)\big)\cong
\indlim~ K_n(D_i).$
\end{center}
\end{proof}

There are two model structures on the category of dendroidal sets: the stable model structure and the 
operadic (Cisinski-Moerdijk) model
structure. The fact that the stable model structure is a left
Bousfield localization of the operadic model structure implies that an operadic
cofibre sequence is also a stable cofibre sequence. In the following by 
cofibre sequence we mean cofibre sequence in the stable model structure and
thereby we also cover the case of cofibre sequences in the Cisinski-Moerdijk model structure. The prototypical example is
induced by a normal inclusion of a dendroidal subset:
\begin{equation*}
D_0 \hookrightarrow D \twoheadrightarrow D/D_0.
\end{equation*}

\begin{proposition}
For a cofibre sequence $X \to Y \to Z$ in $\dSet$ we obtain a fibre sequence
$i^*X_K \to i^*Y_K \to i^*Z_K$ of spaces and thus a 
long exact sequence
\begin{equation*}
\xymatrix{
  \ar[r] &
  K_n(X) \ar[r] & 
  K_n(Y) \ar[r]& 
  K_n(Z) \ar`r[d]`[l]`[llld]`[dll][dll]& \\
 & K_{n-1}(X) \ar[r] & 
 \cdots  \ar[r]& 
 K_1(Z) \ar`r[d]`[l]`[llld]`[dll][dll]& \\ 
 &K_0(X) \ar[r] & K_0(Y) \ar[r] & K_0(Z) \ar[r] & 0 \text{.} & & 
}
\end{equation*}
of $K$-theory groups.
\end{proposition}
\begin{proof}
The first assertion is Corollary 5.5. in \cite{BaNi} and the second is just the
long exact sequence of homotopy groups.
\end{proof}

\begin{example}
Let $D$ be an arbitrary dendroidal set. Then the sequence $\emptyset \to D \to
D$ is a cofibre sequence. Thus the long exact sequence implies that 
$K_n(\emptyset) = 0$ for all $n$. This in particular implies that the morphism
$\emptyset \to *$ is a stable equivalence and the homotopy category is pointed.
\end{example}

\section{The K-theory spectrum}\label{sec_spec}

\newcommand{\Fr}{\textnormal{Fr}}
\newcommand{\Ho}{\mathcal{H}o}
\newcommand{\St}{\widetilde{St}}

In \cite{BaNi} it was shown that the category $\dSet$ together with the stable
model structure is Quillen equivalent to the category of connective spectra.
The proof was based on results of \cite{Heuts1, Heuts2}. 
We do not want to go into the details of the construction here. We only briefly
note a few facts and refer to Appendix \ref{appendix} for more
details. First let $E_\infty$ denote the Barratt-Eccles operad, which is a
simplicial $E_\infty$-operad. The category of algebras for this operad is then
denoted by $\EsSet$ and carries a model structure induced from the model structure on simplicial sets.
 There is a functor 
\begin{equation*}
\widetilde{St}: \dSet \to \EsSet 
\end{equation*}
that is left Quillen and has the property that it maps operadic equivalences to
weak homotopy equivalences of $E_\infty$-spaces. Moreover it induces after
localization on both sides an equivalence of the stable homotopy category of dendroidal sets
to the homotopy-category of grouplike $E_\infty$-spaces. In particular $\St$ sends
stable equivalences to group-completion equivalences. The definition and
properties of the functor $\widetilde{St}$ can be found in Appendix
\ref{appendix} but the reader does not need to know the details of the
construction. They are just needed for the following lemma and the rest is
deduced by abstract reasoning.

\begin{lemma}\label{point}
The functor $\widetilde{St}$ sends the dendroidal set $\eta = i_!\Delta[0]$ to
an $E_\infty$-space which is weakly homotopy equivalent to the free
$E_\infty$-space on one generator. This $E_\infty$-space can be described as the
nerve of the category of finite sets with isomorphisms and tensor product given
by disjoint union. 
\end{lemma}
\begin{proof}
We show this by explicitly computing the functor. By definition, $\eta$ gets send
to $\widetilde{St}(\eta) = St_{E_\infty}(\eta \times \mathcal E_\infty)$ as 
described in the Appendix.
The first thing we use is that $\eta$ is cofibrant as a dendroidal set, and thus
admits a morphism 
$\eta \to \mathcal{E}_\infty$ (for the choices we have made in the Appendix this morphism 
is
actually unique). Thus there is a morphism $\eta \to \eta \times
\mathcal{E}_\infty$ over $\mathcal{E}_\infty$ which is an operadic equivalence
. Therefore we have that
$St_{E_\infty}(\eta) \simeq St(\eta \times \mathcal E_\infty)$ and it only remains to
compute $St(\eta)$. Using the definition of the straightening (see \cite{Heuts1} or Appendix \ref{appendix})  it is easy to see
that $St_{E_\infty}(\eta)$ is the free $E_\infty$-space on one generator.
Finally we note that for $E_\infty$ the Barratt-Eccles Operad we immediately get
\begin{equation*}
\Fr_{E_\infty}(\Delta[0]) \cong \bigsqcup_{n \in \mathbb{N}} B\Sigma_n  \cong N(\text{FinSet})
\end{equation*}
using the usual formula for free algebras over operads.
\end{proof}

\begin{lemma}\phantomsection\label{keylemma}
\begin{enumerate}
\item
For $X \in \sSet$ there is a natural isomorphisms $\widetilde{St}(i_!X) \cong
\Fr_{E_\infty}(X)$ in the homotopy category of $E_\infty$-spaces (with 
weak homotopy equivalences inverted).
\item
For $D \in \dSet$ there is a natural isomorphism in the homotopy category of
simplicial sets from the derived underlying space $i^*D_K$ to the underlying
simplicial set $\Omega B(\widetilde{St}D)$ of the group-completion  of the
$E_\infty$-space $\widetilde{St}D$.
\end{enumerate}
\end{lemma}
\begin{proof}
For the proof of (1) we will use the fact that the homotopy category of
simplicial sets is the universal homotopy theory on a point.
A more precise statement using model categories is that left Quillen functors
from $\sSet$ to any model category $M$ are fully determined on the point
 (see \cite[Proposition 2.3 and Example 2.4]{duggerUniversal}). This means that evaluation on the point induces an equivalence between the homotopy category of left Quillen functors  $\sSet \to M$ and the homotopy category of $M$. 
  In the theory of
$\infty$-categories the universal property is that left adjoint functors 
from the $\infty$-categories of simplicial sets are determined on the point (see
\cite[Theorem 5.1.5.6]{HTT}). This means likewise that there is an equivalence between the $\infty$-category of left adjoint functors (in the $\infty$-sense) from the $\infty$-category of simplicial sets to any presentable $\infty$-category $\mathcal{C}$ and $\mathcal{C}$.

Using one of these two statements we see that the
two left Quillen functors $\sSet \to \EsSet$ given by 
\begin{equation*}
X \mapsto \Fr_{E_\infty}(X) \qquad \text{and} \qquad X \mapsto
\widetilde{St}(i_!X)
\end{equation*}
are isomorphic in the homotopy category if they agree on the point. But this is
the assertion of Lemma \ref{point}.\\

In order to prove the second statement we note that the first part of the lemma implies
that $\widetilde{St}(i_!X)$ and $\Fr_{E_\infty}(X)$ are also equivalent in the
group-completion model structure on $\EsSet$.
In other words the diagram of homotopy categories
\begin{equation*}
\xymatrix{
& \Ho(\sSet) \ar[rd]^{\Fr} \ar[ld]_{i_!} & \\
\Ho(\dSet_{stab}) \ar[rr]^{\widetilde{St}} && \Ho(\EsSet_{grp})
}
\end{equation*}
commutes (up to a natural isomorphism which we suppress). Replacing $\widetilde St$ by its inverse we get an isomorphism of
functors $i_! \cong \widetilde{St}^{-1} \circ \Fr_{E_\infty}$ on the respective
homotopy categories. Thus there is an induced isomorphism for the right adjoint
functors $Ri^* \cong RU^* \circ \widetilde{St}$. 
\begin{equation*}
\xymatrix{
& \Ho(\sSet)  & \\
\Ho(\dSet_{stab}) \ar[ru]^{Ri^*}\ar[rr]^{\widetilde{St}} && \Ho(\EsSet_{grp})\ar[lu]_{RU^*}
}
\end{equation*}

Here $Ri^*$ is the right
derived functor of the underlying space, i.e. given by $i^*D_K$ for a dendroidal
set $D$. The functor $RU^*$ is the right derived functor for the underlying space of an
$E_\infty$-space. Since we are dealing with the group completion model structure
the functor $RU^*$ is given by $\Omega B X$ for an $E_\infty$-space $X$ (or another model of the group completion). Together
this shows that $i^*D_k$ and the space $\Omega B \widetilde{St}(D)$ are
naturally equivalent in the homotopy category of simplicial sets.
\end{proof}

\begin{remark}
The natural isomorphisms constructed above are actually slightly more structured
than just transformations on the homotopy category. The proof shows, mutatis mutandis,  that they
are transformations of $\infty$-functors between $\infty$-categories.
\end{remark}

We have already mentioned that fully Kan dendroidal sets correspond to grouplike
$E_\infty$-spaces (this is the main result in \cite{BaNi}). On the other hand it is well-known
that grouplike $E_\infty$-spaces are essentially the same thing as connective
spectra. By means of a delooping machine we can define a functor 
\begin{equation*}
B^\infty: \EsSet \to \Spec
\end{equation*}
where $\Spec$ is the category of spectra. The functor $B^\infty$ can be chosen such that it
sends group completion equivalences of $E_\infty$-spaces (which are weak homotopy
equivalence after group completion) to stable equivalences and induces after
localization the desired equivalence of grouplike $E_\infty$-spaces to
connective spectra. We do not want to fix a specific choice of delooping machine
or model of spectra but refer the reader to the extensive literature on the topic
\cite{May74,may1977infinite, may1978uniqueness}. It is just important to note
that the spectrum $B^\infty X$ comes with a natural morphism 
\begin{equation*}
X \to \Omega^\infty B^\infty X
\end{equation*}
which is a weak homotopy equivalence if $X$ is grouplike. It follows that this morphism is a group-completion
if $X$ is not grouplike.

\begin{definition}
Let $D$ be a dendroidal set. We define the \emph{K-theory} spectrum of $D$ as
\begin{equation*}
\calK(D) := B^\infty\widetilde{St}(D).
\end{equation*}
This assignment defines  a functor $\calK: \dSet \to \Spec$ that preserves stable equivalences and induces an equivalence between the stable
homotopy categories of dendroidal sets and connective spectra (see \cite[Theorem
5.4]{BaNi}.)
\end{definition}

\begin{thm}\phantomsection\label{prop53}
\begin{enumerate}
\item
For a dendroidal set $D$ the derived underlying space $i^*D_K$ (Definition \ref{def:higher}) 
is naturally homotopy equivalent to $\Omega^{\infty}\calK(D)$. 
In particular $i^*D_K$ is an infinite loop space for
each dendroidal set $D$. Notably all $K_n(D)$ are abelian groups and we have  
\begin{equation*}
K_n(D) \cong \pi_n(\calK(D)). 
\end{equation*}
\item
For a simplicial set $X$ the spectrum $\calK(i_!X)$ is weakly equivalent to the
suspension spectrum $\Sigma^\infty_+X$.
\end{enumerate}
\end{thm}

\begin{proof}
The first statement follows from Lemma \ref{keylemma} and the fact that
$\Omega^{\infty}\calK(D) = \Omega^\infty B^\infty \widetilde{St}(D)$ is homotopy
equivalent to the underlying space of the group completion of
$\widetilde{St}(D)$ by the properties of the delooping machine $B^\infty$.

For the second statement note that the first assertion means that we have an isomorphism
of functors $\Omega^{\infty}\circ \calK \cong Ri^*$ on the homotopy categories.
Since $\calK$ is an equivalence (when restricted to connective spectra) we also have an equivalence 
$\Omega^{\infty} \cong Ri^* \circ \calK^{-1}$. Therefore we get an equivalence
of left adjoint functors $\Sigma^\infty_+ \cong \calK \circ i_!$ where we have
used that $\calK$ is left adjoint to $\calK^{-1}$. 
\end{proof}

\begin{corollary}\label{corollaryeta}
We have $\calK(\eta) \simeq \mathbb{S}$, where $\mathbb{S}$ denotes the sphere
spectrum. Thus
$K_n(\eta) \cong \pi_n^\mathcal{S}$ where $\pi_n^\mathcal{S}$ is the $n$-th
stable homotopy group of spheres. 
\end{corollary}

Now we have computed the K-theory for the object  $\eta \in \dSet$ corresponding to the simplest tree $L_0 = |~$. 
Our next goal is to compute it for the objects $\Omega[T]$ corresponding to arbitrary trees $T$.
Therefore we need the following lemma.

\begin{lemma}\label{leafinclusion}
The morphism 
\begin{equation*}
\bigsqcup_{\ell(T)} \eta \longrightarrow \Omega[T]
\end{equation*}
is a stable equivalence of dendroidal sets. Here $\ell(T)$ denotes the set of
leaves of $T$ and the morphism is given by the associated morphisms $\eta \to
\Omega[T]$ for each leaf.
\end{lemma}
\begin{proof}
We prove this lemma using a result of \cite{MC10}. For each tree we define the
\textit{Segal core} 
\begin{equation*}
Sc[T] := \bigcup_{v} \Omega[C_{n(v)}] ~ \subset ~ \Omega[T]
\end{equation*}
where the union is over all the vertices of $T$, and $n(v)$ is the number of
input edges
at $v$. For $\Omega[T] = \eta$ we put $Sc[T] := \eta$. Then the inclusion $Sc[T]
\to \Omega[T]$ is a weak operadic equivalence \cite[Proposition 2.4]{MC10}. 

Our morphism obviously factors through the Segal core:
\begin{equation*}
\bigsqcup_{\ell(T)} \eta \longrightarrow Sc[T] \longrightarrow \Omega[T].
\end{equation*}
Thus we have to show that the left morphism is a stable equivalence (note that
the right hand morphism is also a stable equivalence since it is a weak operadic
equivalence). We do this by induction over the number $N$ of vertices of $T$.
For $N=1$ we have 
$Sc[T] = \Omega[T] = \Omega[C_n]$ and therefore the morphism is the inclusion of
leaf colours in the corolla which is an outer horn and hence a stable
equivalence.

Now assume the statement is true for $N \in \mathbb{N}$ and $T$ has $N+1$
vertices. 
We pick the vertex $v_0$ at the root with input colours $a_1,...,a_n$. Moreover
we denote by $T_1,...,T_n$ the trees that sit over its leaves (possibly $T_i =
\eta$ if there are not further vertices over $a_i$). Then we have
a pushout diagram
\begin{equation*}
\xymatrix{
\bigsqcup_{i=1}^n \eta \ar[r]\ar[d]_{\sqcup a_i} & \Omega[C_n]\ar[d]^{v_0} \\
\bigsqcup_{i=1}^{n} Sc[T_i] \ar[r] & Sc[T]
}
\end{equation*}
in $\dSet$. The upper horizontal morphism is a trivial stable cofibration and
therefore also the lower horizontal morphism. Thus in the factorization
\begin{equation*}
\bigsqcup_{\ell(T)} \eta \longrightarrow \bigsqcup_{i=1}^{n} Sc[T_i]
\longrightarrow Sc[T]
\end{equation*}
the left hand morphism is a stable equivalence by the induction hypothesis and
the right hand morphism is a stable equivalence as shown above. 
\end{proof}

\begin{corollary}\label{komega}
We have weak equivalences 
\begin{equation*}
\calK\big(\Omega[T]\big)~ \simeq ~\bigvee_{\ell(T)} \mathbb{S} ~ \simeq~
(\mathbb{S})^{\times \ell(T)}
\end{equation*}
and thus $K_n\big(\Omega[T]\big) \cong \bigoplus_{\ell(T)}\pi_n^\mathcal{S}$
where $\ell(T)$ is the set of leaves of $T$ and $\pi_n^\mathcal{S}$ are the
stable homotopy groups of spheres.
\end{corollary}
\begin{proof}
By Lemma \ref{leafinclusion} and the fact that $\calK$ is invariant under stable
equivalences we only need to show that $\calK\big(\bigsqcup \eta\big) \simeq
\bigvee \mathbb{S}$. This is true by Corollary \ref{corollaryeta} together with the fact that $\calK$ is an
equivalence of homotopy-categories and therefore preserves coproducts (which are
computed in connective spectra as in spectra by the wedge sum). 
\end{proof}

\section{Comparison with ordinary algebraic K-Theory}

\newcommand{\cald} {\mathcal{D}}

In this section we want to show how algebraic $K$-theory of dendroidal sets generalizes classical
algebraic $K$-theory of rings. 

Therefore recall that by definition algebraic $K$-theory of a ring $R$ is
computed using its groupoid of finitely generated projective modules (or some
related space like $BGl(R)$). There are several equivalent variants to produce a
spectrum $\mathcal{K}(R)$ from this category. Let us recall an easy
`group-completion' variant here. 

The variant we want to describe does not only work for the category of
finitely generated projective modules over a ring but more generally for an
arbitrary symmetric monoidal category $C$.
The first step is to take the nerve $NC$ which is an $E_\infty$-space and then
apply a delooping machine to get a $K$-theory spectrum $\calK(C) := B^\infty NC$
(see e.g. \cite{th82}). The main purpose of this section is to prove the
following theorem:

\begin{thm}\label{generalization}
For a symmetric monoidal groupoid $C$ we have a weak equivalence of spectra
\begin{equation*}
 \calK\big(N_d C\big) \simeq \calK(C)
\end{equation*}
and therefore also $K_n(N_d C) \cong K_n(C)$. 
\end{thm}

\begin{remark}
We restrict ourselves here to the case of symmetric monoidal groupoids but the
theorem holds for arbitrary symmetric monoidal categories $C$. However groupoids
are the most important case if we are interested in algebraic $K$-theory of
rings. The reason for our restriction is that the proof
of the theorem as it is stated here is relatively easy and formal whereas in the
case of arbitrary symmetric monoidal categories one has to use explicit models
and the calculation becomes long and technical.
\end{remark}

\begin{corollary}
Let $R$ be a ring and $\text{Proj}_R$ denote the groupoid of finitely generated,
projective $R$-modules. Then we have 
\begin{equation*}
\calK\big(N_d\text{Proj}_R \big) \simeq \calK(R)
\end{equation*}
and therefore $K_n(N_d\text{Proj}_R) = K_n(R)$ where $K_n(R)$ are the algebraic
K-theory groups of $R$.
\end{corollary}

\begin{lemma}
Let $\Symg$ denote the category of small symmetric monoidal
group\-oids (which is really a 2-category). Then the functors 
\begin{equation*}
N: \Symg \to \EsSet \qquad \text{and} \qquad  N_d: \Symg \to \dSet_{cov}
\end{equation*}
both admit left adjoint functors at the level of homotopy-categories (resp. $\infty$-categories). 
Here we take the covariant model structure on dendroidal sets.
\end{lemma}
\begin{proof}
We explicitly describe both left adjoint functors. The first is well known and the second is 
a variant of a functor considered by Moerdijk-Weiss \cite[section 4]{MoerW07}.

The left adjoint of $N: \Ho(\Symg) \to \Ho(\EsSet) $ is the `1-truncation'
functor $\tau$. Let $X$ be an $E_\infty$-space, i.e. a simplicial set with the
structure of an algebra over the Barratt-Eccles operad. Then $\tau(X)$ is the free
groupoid generated by 0-simplices of $X$ as objects and 1-simplices of $X$ as
morphisms subject to the relations generated by 2-simplices of $X$.
The groupoid $\tau(X)$ inherits the structure of a
symmetric monoidal category given on generators by the $E_\infty$-structure of
$X$. The assignment $X \mapsto \tau(X)$ even defines a functor 
\begin{equation*}
\tau: \EsSet \to \Symg
\end{equation*}
which sends weak homotopy equivalences of $E_\infty$-spaces to equivalences of symmetric
monoidal groupoids and which is left adjoint to $N$ as functors of
$\infty$-categories, hence also on the level of homotopy categories.

The left adjoint of $N_d: Ho(\Symg) \to Ho(\dSet_{cov})$ can be described
very similarly. For a dendroidal set $D$ it is given by the free symmetric
monoidal groupoid generated by the objects of $D$ and for
each $n$-corolla from $d_0,...,d_n$ to $d$ there is an isomorphism $d_0 \otimes
... \otimes d_n \to d$. The relations are given by corollas with two vertices.
This is just the symmetric monoidal groupoid version of the functor $\tau_d$
described in \cite[section 4]{MoerW07}. We denote this functor by 
\begin{equation*}
\tau_d^{\otimes}: \dSet \to \Symg
\end{equation*} 
It is easy to see that this functor descends to the level of infinity categories. By construction
it is clear that the functor is then an adjoint on the level of $\infty$-categories.
\end{proof}

\begin{proof}[Proof of Theorem \ref{generalization}.] 
The proof is based on the fact that both K-theory spectra are defined 
similarly. The K-theory spectrum $\calK(N_dC)$ is defined as the spectrum associated to
the $E_\infty$-space $\widetilde{St}(N_dC)$ and the K-theory spectrum of $C$ as the spectrum associated to
the $E_\infty$-space $NC$. Thus we only have to show that for a given symmetric monodical groupoid there is a (natural)
equivalence of $E_\infty$-spaces $\widetilde{St}(N_dC) \cong NC$. In other words
we want to show that the diagram
\begin{equation}\label{diagram}
\xymatrix{
 & \Ho(\Symg) \ar[ld]_{N_d}\ar[rd]^{N} & \\
\Ho(\dSet_{cov}) \ar[rr]^{\widetilde{St}} && \Ho(\EsSet)
}\end{equation}
commutes. Here $\Symg$ denotes the category of symmetric monoidal
groupoids and $\dSet_{cov}$ indicates that we are working with the covariant model structure on dendroidal sets. In this case
the lower horizontal functor is an equivalence.\\

In order to show the commutativity of the above diagram we use the fact that all
functors in this diagram admit left adjoints. The left adjoints of the upper two
functors are described in the last lemma and 
the left adjoint of $\widetilde{St}$ is the inverse $\widetilde{St}^{-1}$. The
commutativity of the above diagram \eqref{diagram} translates into a natural
equivalence $\tau_d^{\otimes} \circ \widetilde{St}^{-1} \cong \tau$ of 
left adjoint functors 
\begin{equation}\label{diagram22}
\xymatrix{
 & \Ho(\Symg)  & \\
\Ho(\dSet_{cov}) \ar[ru]^{\tau^\otimes_d} && \Ho(\EsSet)\ar[lu]_{\tau}\ar[ll]_{\widetilde{St}^{-1}}
}\end{equation}

As a next step we want to use the universal property of the $\infty$-category of
$E_\infty$-spaces established in \cite{Gep}: it is the free presentable, preadditive $\infty$-category on
one generator.
Let us explain what this means. First an $\infty$-category $\cald$ is 
called preadditive if finite products and coproducts agree, more precisely if it is pointed and the canonical map from the coproduct to the product is an equivalence. Let $\cald$ be a presentable,
preadditive $\infty$-category. Then the universal property of $\EsSet$ is that 
there is an equivalence of $\infty$-categories
\begin{equation*}
\text{Fun}^L(\EsSet, \cald) \simeq \cald
\end{equation*}
where $\text{Fun}^L$ denotes the category of left adjoint functors. The 
equivalence is given by evaluation on the free $E_\infty$-algebra on one
generator $\text{F} := \Fr_{E_\infty}(\Delta[0])$. In particular if we have two left 
adjoint functors $A,B:\EsSet \to \cald$ then they are naturally equivalent if and only if there is an
equivalence of $A(\text{F}) \to B(\text{F})$ in $\cald$. 

In our case we first observe that $\Symg$ is clearly preadditive. Thus in
order to compare the left adjoint functors  $\tau_d^{\otimes} \circ
\widetilde{St}^{-1}$ and $\tau$ we only have to show that they assign equivalent
symmetric monoidal groupoids to the free $E_\infty$-algebra on the point. Now
the right functor $\tau$ is easy to evaluate since $\text{F} = \Fr_{E_\infty}(\Delta[0])$ is the nerve of the
category of finite sets with isomorphisms. Thus $\tau( \text{F})$ is the category of
finite sets with isomorphisms and tensor product given by disjoint union.
Observe that this is also the free symmetric monoidal category on one generator.

The functor 
$\widetilde{St}^{-1}$ evaluated on $\text{F}$ is given by $\eta$ as shown in Lemma
\ref{point}. Thus $\tau_d \circ \widetilde{St}^{-1}$ of $\text{F}$ is given by the free
symmetric monoidal category on one generator and hence equivalent to $\tau(F)$. 
Together we have shown that the
two functors $\tau_d \circ \widetilde{St}^{-1}$ and $\tau$ agree on $\text{F}$ and together with the universal property of
$\EsSet$ this completes the proof.
 
\end{proof}

\section{Appendix: the straightening functor of Heuts} \label{appendix}

In this appendix we want to collect some facts about the dendroidal
straightening functor which has been defined and studied by Heuts
\cite{Heuts1}. 
We do not describe the most general instance of the straightening functor, but
only the variant which is used in this paper. First we have to fix some notation. By
$E_\infty$ we denote the Barratt-Eccles operad. This is a simplicially
enriched operad with one colour and the simplicial set of operations is given by
\begin{equation*}
E_\infty(n) = E\Sigma_n = \Sigma_n // \Sigma_n \ .
\end{equation*}
Here $\Sigma_n$ is the permutation group on $n$ letters and $\Sigma_n //
\Sigma_n$ is the simplicial set given as the nerve of the action groupoid of
$\Sigma_n$ on itself by right multiplication. 

Now there is an adjunction between dendroidal sets and simplicially enriched
operads
\begin{equation} \label{nerve_adj}
hc\tau_d: \xymatrix{\dSet \ar@<0.3ex>[r] & \sOper: hcN_d \ar@<0.7ex>[l]}
\end{equation}
which was introduced in \cite{MoerW07} and studied in much more detail in 
 \cite{MC11}. 
 The left adjoint is completely determined on representables since $\dSet$ is a presheaf category. 
 The definition of $hc\tau_d$ is 

\begin{equation*}
hc\tau_d(\Omega[T]) := W\Omega(T)
\end{equation*}
where $W$ is the Boardman-Vogt resolution of operads \cite{BoarVogt}. Let us
describe the simplicially enriched operad $W\Omega(T)$ explicitly here:
Its colours are the edges of $T$. For edges $c_1,...,c_n,c$ the simplicial set
of operations $W\Omega(T)(c_1,...,c_n; c)$ is given by
\begin{equation*}
W\Omega(T)(c_1,...,c_n; c) =  \Delta[1]^{i(V)} 
\end{equation*}
if there is a maximal subtree $V$ of $T$ with leaves $c_1,...,c_n$ and root $c$.
In this case $i(V)$ denotes the number of inner edges of the uniquely determined
subtree $V$. If there is no such subtree 
we set $W\Omega(T)(c_1,...,c_n; c) =  \emptyset$. The composition in
$W\Omega(T)$ is
given by grafting trees, assigning length 1 to the newly arising inner edges.
See Remark 7.3 of \cite{MW09} for a more detailed description of the
composition. 

There is a distinguished algebra $A_T$ for the operad $W\Omega(T)$ which will
play an important role in the definition of the straightening later. The value
of $A_T$ on the colour $c$ (which is an edge of $T$) is given by:
\begin{equation*}
A_T(c) := \Delta[1]^{i(c)}
\end{equation*}
where $i(c)$ is the number of edges over $c$ in $T$. The structure maps of $A_T$
as an $W\Omega(T)$-algebra
\begin{equation*}
W\Omega(T)(c_1,...,c_n; c) \times A_T(c_1) \times ... \times A_T(c_n)  \to
A_T(c)
\end{equation*}
are given by grafting trees, assigning length 1 to the newly arising inner edges
$c_1, . . . , c_n$. \\

In fact Moerdijk and Cisinski have shown that the adjunction \eqref{nerve_adj} is
a Quillen equivalence between the Cisinski-Moerdijk model structure on $\dSet$
and an appropriate model structure on simplicially enriched, coloured operads.
We now apply the homotopy coherent dendroidal nerve to the Barratt-Eccles operad
and obtain a dendroidal set 
\begin{equation*}
\mathcal{E}_\infty := hcN_d(E_\infty).
\end{equation*}
Let $\dSet/\mathcal{E}_\infty$ denote the category of dendroidal sets over
$\mathcal{E}_\infty$. The category of $\EsSet$ is the category of algebras for
the Barratt-Eccles operad in the category of simplicial sets (so space here
 means simplicial set). 

We are now ready to give the definition of the straightening functor\begin{equation*}
St_{E_\infty}:  \dSet / \mathcal{E}_\infty \to {\EsSet}\  .
\end{equation*}
The original definition of the straightening functor (which is a significantly more general variant)
 is given in \cite[Section 2.2]{Heuts1}. First we remark that the category
$\dSet / \mathcal{E}_\infty$ is freely generated under colimits by objects of
the form
$\Omega[T] \stackrel{s}{\to} \mathcal{E}_\infty$, where  $T$ is a tree and 
$s$ and arbitrary morphism.  
We will define $St_{E_\infty}$ for those objects
and then left Kan extend it to the whole category $\dSet /
\mathcal{E}_\infty$.

In order to define $St_{\mathcal{E}_\infty}\big(\Omega[T] \stackrel{s}\to
\mathcal{E_\infty})$ we use the fact that by adjunction the morphism $s: \Omega[T] \to
\mathcal{E}_\infty$ uniquely determines a morphism $\tilde{s}: W\Omega[T] \to
E_\infty$. This morphism $\tilde{s}$ then induces an adjunction
\begin{equation*}
\xymatrix{\tilde{s}_!: ~  W\Omega(T)\text{-}\mathcal{S}\text{paces} \ar@<0.3ex>[r] & \EsSet
~:\tilde{s}^*  \ar@<0.7ex>[l]}
\end{equation*}
where the right adjoint $\tilde{s}^*$ is given by pullback along $\tilde{s}$.
Finally we can define 
\begin{equation*}
St_{E_\infty}\big(\Omega[T] \stackrel{s}\to \mathcal{E_\infty}) :=
\tilde{s}_!(A_T).
\end{equation*} 
where $A_T$ is the $W\Omega(T)$-algebra defined above. Together this defines the
desired functor $St_{E_\infty}$ by left Kan extension. \\

The straightening functor as defined above has nice homotopical properties (see
the next proposition). But before we describe these properties, we want to get
rid of the overcategory $\dSet/\mathcal{E}_\infty$. There is an easy way to do
this, by considering the functor $\dSet \to \dSet/\mathcal{E}_\infty$ which is
defined by $D \mapsto D \times \mathcal{E}_\infty$. It turns out that this
functor has both adjoints. Replacing a dendroidal set $D$ by the dendroidal set
$D \times \mathcal{E}_\infty$ amounts to cofibrant replacement. This is
basically because $\mathcal{E}_\infty$ is a cofibrant resolution of the point.
Thus homotopically the category $\dSet/\mathcal{E}_\infty$ is equivalent to $\dSet$
itself. Therefore we define the functor $\widetilde{St}$ as the composition
$\dSet \to \dSet/\mathcal{E}_\infty \to \EsSet$, concretely $\widetilde{St}(D)
:= St_{E_\infty}(D \times \mathcal{E}_\infty)$. 
\begin{prop}[Heuts]
\begin{enumerate}
\item The functors
\begin{equation*}
\widetilde{St}:  \dSet \to \EsSet\qquad \text{and} \qquad St_{E_\infty}: 
\dSet/\mathcal{E}_\infty \to \EsSet 
\end{equation*}
are left Quillen with respect to the Moerdijk-Cisinski model structure and the
usual model structure on $\EsSet$ (in which the fibrations and weak equivalences
are taken to be those of the underlying simplicial sets). Moreover both functors
send operadic equivalences to weak homotopy equivalences.

\item There are further model structures on $\dSet$ and
$\dSet/\mathcal{E_\infty}$ called the covariant model structures. For these
model structure the above functors are Quillen equivalences.

\end{enumerate}
\end{prop}

\begin{variant}
There are variants of the straightening functor for different choices of
$E_\infty$-operads than the Barratt-Eccles operad. More precisely if $E$ is any
simplicial $E_\infty$-operad which is cofibrant, then there is a corresponding
straightening functor $$St_E: \dSet/hcN_d(E) \to E\text{-spaces}$$ Conversely
for any choice of cofibrant resolution of the point $\mathcal{E} \to *$ in
dendroidal sets, there is a straightening functor $$St: \dSet/\mathcal{E} \to
hc\tau_d(E)\text{-spaces}$$ These functors are defined with the
same formulas as above and are all essentially equivalent to the variant using
the Barratt-Eccles operad.
\end{variant}

\bibliographystyle{alpha}
\bibliography{Alles}

\end{document}